\documentclass[a4paper,11pt]{amsart}

\usepackage[margin=1in,includehead,includefoot]{geometry}
\usepackage{amsthm,amsmath,amssymb}
\usepackage{mathrsfs}
\usepackage{enumitem}
\usepackage[pdftitle={Asymmetric coloring games on incomparability graphs},
            pdfauthor={Tomasz Krawczyk and Bartosz Walczak},
            colorlinks=true,linkcolor=black,citecolor=black,filecolor=black,urlcolor=black]{hyperref}

\renewenvironment{enumerate}{\begin{enumorig}[label=\textup{(\arabic*)}, noitemsep, topsep=3pt plus 3pt, leftmargin=*]}{\end{enumorig}}
\newenvironment{enumerates}{\begin{enumorig}[label=\textup{($*$)}, noitemsep, topsep=3pt plus 3pt, leftmargin=*]}{\end{enumorig}}

\renewenvironment{itemize}{\begin{itemorig}[label=\textbullet, noitemsep, topsep=3pt plus 3pt, leftmargin=1.1em]}{\end{itemorig}}

\newtheorem{thm}{Theorem}
\newtheorem{conj}[thm]{Conjecture}
\newtheorem{lem}[thm]{Lemma}

\def\clF{\mathscr{F}}
\def\clG{\mathscr{G}}
\def\clI{\mathscr{I}}
\def\clP{\mathscr{P}}
\def\clQ{\mathscr{Q}}
\def\clT{\mathscr{T}}

\def\posP{\mathcal{P}}
\def\posQ{\mathcal{Q}}

\def\scrI{\mathscr{I}}

\let\leq\leqslant
\let\geq\geqslant
\let\setminus\smallsetminus
\let\Gamma\varGamma
\let\delete\xi

\DeclareMathOperator{\col}{col}

\def\valg{\chi_{\mathrm{g}}}
\def\vala{\chi_{\mathrm{a}}}
\def\grg{\Gamma_{\mathrm{g}}}
\def\gra{\Gamma_{\mathrm{a}}}
\def\colg{\col_{\mathrm{g}}}

\title{Asymmetric coloring games on incomparability graphs}

\makeatletter
\g@addto@macro{\@settitle}{\begin{center}\bfseries extended abstract\end{center}}
\makeatother

\author{Tomasz Krawczyk\and Bartosz Walczak}

\thanks{Theoretical Computer Science Department, Faculty of Mathematics and Computer Science, Jagiellonian University, Krak\'ow, Poland; \textit{E-mail addresses}: \texttt{krawczyk@tcs.uj.edu.pl}, \texttt{walczak@tcs.uj.edu.pl}}
\thanks{The authors were supported by Polish National Science Center grant 2011/03/B/ST6/01367.}

\begin{document}

\begin{abstract}
Consider the following game on a graph $G$: Alice and Bob take turns coloring the vertices of $G$ properly from a fixed set of colors; Alice wins when the entire graph has been colored, while Bob wins when some uncolored vertices have been left.
The \emph{game chromatic number} of $G$ is the minimum number of colors that allows Alice to win the game.
The \emph{game Grundy number} of $G$ is defined similarly except that the players color the vertices according to the first-fit rule and they only decide on the order in which it is applied.
The \emph{$(a,b)$-game chromatic} and \emph{Grundy numbers} are defined likewise except that Alice colors $a$ vertices and Bob colors $b$ vertices in each round.
We study the behavior of these parameters for incomparability graphs of posets with bounded width.
We conjecture a complete characterization of the pairs $(a,b)$ for which the $(a,b)$-game chromatic and Grundy numbers are bounded in terms of the width of the poset; we prove that it gives a necessary condition and provide some evidence for its sufficiency.
We also show that the game chromatic number is not bounded in terms of the Grundy number, which answers a question of Havet and Zhu.
\end{abstract}

\maketitle

\vspace{-3.2ex}

\section{Introduction}

\subsection*{Definitions}
A \emph{proper coloring} of a graph $G$ is an assignment of colors to the vertices of $G$ such that no two adjacent ones have the same color.
The minimum number of colors in a proper coloring of $G$ is called the \emph{chromatic number} of $G$ and denoted by $\chi(G)$.
Assuming that the colors are ordered, a simple procedure to construct a proper coloring of $G$ is to process the vertices of $G$ one by one and let each vertex $v$ be assigned the least color that has not been used before on any neighbor of $v$.
Such a proper coloring of $G$ is called a \emph{first-fit coloring} and depends on the order in which the vertices of $G$ are processed.
Although there is always an order of the vertices of $G$ for which the first-fit coloring is optimal, it can use a lot more than $\chi(G)$ colors in general.
The maximum number of colors in a first-fit coloring of $G$ over all orders of the vertices of $G$ is called the \emph{Grundy number} of $G$ and denoted by $\Gamma(G)$.
The \emph{coloring number} of $G$, denoted by $\col(G)$, is the least integer $k$ for which there is a linear order of the vertices of $G$ such that every vertex has fewer than $k$ neighbors earlier in the order.
We have $\chi(G)\leq\col(G)$, as the first-fit coloring of $G$ according to the order witnessing $\col(G)$ uses at most $\col(G)$ colors.

The game variants of the parameters defined above are defined through games played on the graph by two players, Alice and Bob.
In the \emph{coloring game} on $G$, Alice and Bob take turns coloring vertices of $G$ properly from a fixed set of colors until one of the players cannot make a legal move; Alice wins when the entire graph has been colored, while Bob wins when some uncolored vertices have been left.
The \emph{game chromatic number} of $G$, denoted by $\valg(G)$, is the minimum number of colors for which Alice has a winning strategy in the coloring game on $G$.
In the \emph{Grundy game} on $G$, Alice and Bob take turns choosing vertices of $G$, which are then colored by the first-fit procedure in the order they have been chosen.
The \emph{game Grundy number} of $G$, denoted by $\grg(G)$, is the least integer $k$ for which Alice can ensure that the resulting first-fit coloring of $G$ uses at most $k$ colors.
In the \emph{marking game} on $G$, Alice and Bob take turns marking vertices of $G$.
The \emph{game coloring number} of $G$, denoted by $\colg(G)$, is the least integer $k$ for which Alice can ensure that every vertex has fewer than $k$ marked neighbors at the time it is marked.
It is clear that $\valg(G),\grg(G),\col(G)\in[\chi(G),\colg(G)]$ and $\grg(G)\leq\Gamma(G)$.

The three games defined above have asymmetric variants, the \emph{$(a,b)$-coloring game}, the \emph{$(a,b)$-Grundy game}, and the \emph{$(a,b)$-marking game}, respectively, which differ from the ordinary ones in that Alice colors/chooses/marks $a$ vertices and Bob colors/chooses/marks $b$ vertices in each round.
The \emph{$(a,b)$-game chromatic number} of $G$, denoted by $\valg(a,b,G)$, the \emph{$(a,b)$-game Grundy number} of $G$, denoted by $\grg(a,b,G)$, and the \emph{$(a,b)$-game coloring number}, denoted by $\colg(a,b,G)$, are defined like $\valg(G)$, $\grg(G)$, and $\colg(G)$, respectively, but using the asymmetric variants of the games.
It follows that $\valg(1,0,G)=\grg(1,0,G)=\chi(G)$, $\grg(0,1,G)=\Gamma(G)$, $\colg(1,0,G)=\col(G)$, $\valg(1,1,G)=\valg(G)$, $\grg(1,1,G)=\grg(G)$, and $\colg(1,1,G)=\colg(G)$.

For a class of graphs $\clG$, let $\valg(\clG)$, $\grg(\clG)$, and $\colg(\clG)$ denote the maximum of $\valg(G)$, $\grg(G)$, and $\colg(G)$, respectively, over all graphs $G\in\clG$, and similarly for the asymmetric variants.

\subsection*{History}
The coloring game was first introduced by Steven J. Brams, see \cite{Gar81}, in the context of colorings of planar maps.
Later it was reinvented by Bodlaender \cite{Bod91}, who defined the game chromatic number and asked whether it is finite for the class of planar graphs.
The problem of estimating the game chromatic number gained a lot of attention and was studied not only for planar graphs but also for various other classes of graphs.
The marking game was explicitly introduced by Zhu \cite{Zhu00} as a tool for bounding the game chromatic number from above, although it was used implicitly already before.

The following benchmark classes of graphs were considered for the game chromatic and coloring numbers: forests $\clF$, outerplanar graphs $\clQ$, planar graphs $\clP$, interval graphs $\clI_k$ with clique number $k$, and partial $k$-trees $\clT_k$.
Bodlaender \cite{Bod91} proved $4\leq\valg(\clF)$.
Faigle, Kern, Kierstead, and Trotter \cite{FKKT93} proved $\valg(\clF)=\colg(\clF)=4$ and $2k\leq\valg(\clI_k)\leq\colg(\clI_k)=3k-2$.
Kierstead and Tuza \cite{KT96}, Guan and Zhu \cite{GZ99}, and Kierstead and Yang \cite{KY05} proved $6\leq\valg(\clQ)\leq\colg(\clQ)=7$.
Wu and Zhu \cite{WZ08} and Zhu \cite{Zhu00} proved $\colg(\clT_k)=3k+2$.
Kierstead and Trotter \cite{KT94} proved $\valg(\clP)\leq 33$, which was later improved many times: to $\valg(\clP)\leq 30$ by Dinski and Zhu \cite{DZ99}, to $\colg(\clP)\leq 19$ by Zhu \cite{Zhu99}, to $\colg(\clP)\leq 18$ by Kierstead \cite{Kie00}, and finally to $\colg(\clP)\leq 17$ by Zhu \cite{Zhu08}.
From below, Kierstead and Tuza \cite{KT96} proved $8\leq\valg(\clP)$, and Wu and Zhu \cite{WZ08} proved $11\leq\colg(\clP)$.
Kierstead \cite{Kie00} devised a powerful \emph{activation strategy} that gives all the above-mentioned upper bounds on the game coloring number except the bound $\colg(\clP)\leq 17$, which was obtained using an enhancement of the activation strategy.

The asymmetric coloring and marking games were introduced by Kierstead \cite{Kie05}, who determined $\valg(a,b,\clF)$ and $\colg(a,b,\clF)$ for all $a$ and $b$.
Kierstead and Yang \cite{KY05} devised a \emph{harmonious strategy} for Alice that gives very strong bounds on the $(a,b)$-game coloring numbers for graphs admitting an edge orientation with maximum in-degree at most $\frac{a}{b}$.
They used it to provide upper bounds on asymmetric coloring numbers for interval, chordal, outerplanar, and planar graphs.
The bounds related to interval and chordal graphs were further improved by Yang and Zhu \cite{YZ08}, who devised a method generalizing both the activation and the harmonious strategy.

The Grundy game was introduced recently by Havet and Zhu \cite{HZ13}, who showed that $\grg(\clF)=3$ and $\grg(\clQ)\leq 7$.
To our knowledge, the asymmetric Grundy games were not considered yet.

\subsection*{Results}
We investigate the asymmetric game chromatic number and the asymmetric game Grundy number in the class of incomparability graphs of partially ordered sets (posets).
The \emph{incomparability graph} of a poset $\posP=(P,\leq)$ is the graph with vertex set $P$ in which two points of $P$ are connected by an edge if and only if they are $\leq$-incomparable.
The \emph{width} of a poset $\posP$ is the maximum size of an antichain in $\posP$, which is the maximum size of a clique in the incomparability graph of $\posP$.
For a poset $\posP$, let $\valg(a,b,\posP)$ and $\grg(a,b,\posP)$ denote the $(a,b)$-game chromatic number and the $(a,b)$-game Grundy number, respectively, of the incomparability graph of $\posP$.
A proper coloring of the incomparability graph of $\posP$ with $k$ colors is equivalent to a partition of $\posP$ into $k$ chains, and by Dilworth's theorem, the chromatic number of the incomparability graph of $\posP$ is equal to the width of $\posP$.
From now on, we assume that the $(a,b)$-coloring game and the $(a,b)$-Grundy game are played directly on the poset $\posP$ so that the points of each color form a chain in $\posP$ and the other rules of the games are applied accordingly.

Our main concern is how large $\valg(a,b,\posP)$ and $\grg(a,b,\posP)$ can be in comparison to the width of $\posP$ for different values of $a$ and $b$.
For $w\geq 1$, let $\valg(a,b,w)$ and $\grg(a,b,w)$ denote the~maximum of $\valg(a,b,\posP)$ and $\grg(a,b,\posP)$, respectively, over all posets $\posP$ of width $w$.
Since the incomparability graphs of posets of width $2$ have unbounded coloring number, the game coloring number is not interesting in this context, and the activation and harmonious strategies give no useful bounds.

\begin{conj}
\label{conj:chromatic_game}
Let\/ $w\geq 2$.
We have\/ $\valg(a,b,w)<\infty$ if and only if\/ $\frac{a}{b}\geq 2$.
\end{conj}

\begin{conj}
\label{conj:Grundy_game}
Let\/ $w\geq 2$.
We have\/ $\grg(a,b,w)<\infty$ if and only if\/ $\frac{a}{b}\geq 1$.
\end{conj}

\noindent
We show that the `only if' parts of our conjectures are true.

\begin{thm}
\label{thm:Bob_infinite_strategy_coloring}
If\/ $w\geq 2$ and\/ $\frac{a}{b}<2$, then\/ $\valg(a,b,w)=\infty$.
\end{thm}

\begin{thm}
\label{thm:Bob_infinite_strategy_Grundy}
If\/ $w\geq 2$ and\/ $\frac{a}{b}<1$, then\/ $\grg(a,b,w)=\infty$.
\end{thm}

\noindent
We also show that if $\frac{a}{b}\geq 2$, then there are posets on which any winning strategy of Alice in the $(a,b)$-coloring game uses exponentially many colors with respect to the width.

\begin{thm}
\label{thm:exponential_lower_bound}
If\/ $w\geq 2$ and\/ $\frac{a}{b}\geq 2$, then\/ $\valg(a,b,w)\geq(1+1/\lfloor\frac{a}{2b}\rfloor)^{w-1}$.
\end{thm}

\noindent
In particular, $\valg(2,1,w)\geq 2^{w-1}$.
We provide an almost matching upper bound for this case.

\begin{thm}
\label{thm:upper_bound}
$\valg(2,1,w)\leq w2^{w-1}$.
\end{thm}

Finally, we address a question raised by Havet and Zhu \cite[Problem 4]{HZ13} whether there is a function $f$ such that all graphs $G$ satisfy $\valg(G)\leq f(\grg(G))$.
The following and the fact that $\grg(G)\leq\Gamma(G)$ imply that the answer is no.

\begin{thm}
\label{thm:Havet_Zhu}
For every\/ $k$, there is a poset\/ $\posP$ such that\/ $\valg(1,1,\posP)>k$ and\/ $\Gamma(\posP)\leq c$, where\/ $c$ is an absolute constant.
\end{thm}

Theorem \ref{thm:upper_bound} is a strong premise that Conjecture \ref{conj:chromatic_game} is true.
Conjecture \ref{conj:Grundy_game} is more mysterious; in particular, we can only prove $\grg(1,1,w)<\infty$ for $w\leq 3$.
The following two results illustrate the nature of the problem; they are also used in the proofs of Theorems \ref{thm:Bob_infinite_strategy_Grundy} and \ref{thm:Havet_Zhu}, respectively.

\begin{thm}[Kierstead \cite{Kie81}] 
\label{thm:First-Fit-infinity}
For every\/ $k$, there is a poset\/ $\posP$ of width\/ $2$ such that\/ $\Gamma(\posP)\geq k$. 
\end{thm}

\begin{thm}[Bosek at al.\ \cite{BKM13}] 
\label{thm:First-Fit-bounded}
For every poset\/ $\posQ$ of width\/ $2$, there is a function\/ $f_\posQ$ such that every poset\/ $\posP$ of width\/ $w$ that does not contain\/ $\posQ$ as an induced subposet satisfies\/ $\Gamma(\posP)\leq f_\posQ(w)$.
\end{thm}

\section{Lower bounds}
\label{sec:lower_bound}

We define the \emph{auxiliary\/ $(a,b)$-coloring game} and the \emph{auxiliary\/ $(a,b)$-Grundy game} like the $(a,b)$-coloring game and the $(a,b)$-Grundy game, respectively, except that
\begin{itemize}
\item Bob starts the game and colors (chooses) $b$ vertices in each round,
\item Alice responds to Bob's moves by coloring (choosing) any number of vertices between $0$ and~$a$.
\end{itemize}
For the auxiliary games, we define parameters $\vala(a,b,\posP)$, $\gra(a,b,\posP)$, $\vala(a,b,w)$, and $\gra(a,b,w)$ analogously to the respective parameters of the ordinary games.
When $x$ is a point of a poset $\posP$, we let $I_\posP(x)$ denote the set of points of $\posP$ that are incomparable to $x$.

\begin{lem}
\label{lem:auxiliary_coloring_game}
For every poset\/ $\posQ$, if\/ $\posP$ is the poset that consists of\/ $n$ copies of\/ $\posQ$ made comparable to each other and\/ $n$ is large enough in terms of\/ $a$, $b$ and\/ $\posQ$, then\/ $\valg(a,b,\posP)\geq\vala(\lfloor\frac{a}{b}\rfloor,1,\posQ)$.
Consequently, $\valg(a,b,w)\geq\vala(\lfloor\frac{a}{b}\rfloor,1,w)$.
Similarly, $\grg(a,b,w)\geq\gra(\lfloor\frac{a}{b}\rfloor,1,w)$.\footnote{We are grateful to Grzegorz Matecki for his observation that the lemma can be applied to the Grundy game.}
\end{lem}

\begin{proof}[Proof idea]
Bob plays his winning strategy of the game on $\posQ$ in parallel on many copies of $\posQ$.
In each round, Alice can `invalidate' some of the copies by coloring more than $\lfloor\frac{a}{b}\rfloor$ vertices in them, but Bob is able to carry on his strategy up to a winning position in at least one copy.
\end{proof}

\begin{lem}
\label{lem:1_1_coloring_lower_bound}
Let\/ $\posP$ be a poset that consists of two chains\/ $C$ and\/ $C'$ such that for every interval\/ $I$ of\/ $C$ containing the minimum or the maximum point of\/ $C$, the chain\/ $C'$ contains\/ $2k$ points\/ $x$ with\/ $I_\posP(x)\cap C=I$.
If\/ $|C|\gg k$, then\/ $\vala(1,1,\posP)>k$.
Consequently, $\vala(1,1,2)=\infty$.
\end{lem}

\begin{proof}[Proof idea]
Bob maintains a set $A$ of colored points of $C'$ and an interval $I\subseteq\bigcap_{x\in A}I_\posP(x)$ of uncolored points of $C$ so that in each round, a point with a new color is added to $A$ and $I$ does not get too much smaller.
He wins when all colors have been used on $A$ and still $I\neq\emptyset$.
\end{proof}

\begin{proof}[Proof of Theorem \ref{thm:Bob_infinite_strategy_coloring}]
If $w\geq 2$ and $\frac{a}{b}<2$, then $\valg(a,b,w)\geq\vala(\lfloor\frac{a}{b}\rfloor,1,w)\geq\vala(1,1,2)=\infty$, by Lemma \ref{lem:auxiliary_coloring_game}, the fact that $\valg(0,1,w)\geq\valg(1,1,w)\geq\valg(1,1,2)$, and Lemma \ref{lem:1_1_coloring_lower_bound}.
\end{proof}

\begin{proof}[Proof of Theorem \ref{thm:Bob_infinite_strategy_Grundy}]
If $w\geq 2$ and $\frac{a}{b}<1$, then $\grg(a,b,w)\geq\gra(\lfloor\frac{a}{b}\rfloor,1,w)\geq\gra(0,1,2)=\infty$, by Lemma \ref{lem:auxiliary_coloring_game}, the fact that $\gra(0,1,w)\geq\gra(0,1,2)$, and Theorem \ref{thm:First-Fit-infinity}.
\end{proof}

\begin{proof}[Proof sketch of Theorem \ref{thm:Havet_Zhu}]
Let $k\geq 1$, $\posQ$ be the poset defined in Lemma \ref{lem:1_1_coloring_lower_bound} for $k$, and $\posP$ be the poset defined in Lemma \ref{lem:auxiliary_coloring_game} for $\posQ$ and $n$ large enough.
Hence the width of $\posP$ is $2$ and $\valg(1,1,\posQ)>k$.
Let $(R,\leq_R)$ be the poset defined so that $R=\{r_1,\ldots,r_{10}\}$ and $r_i\leq_Rr_j$ if and only if $i+1<j$.
It is easy to verify that $\posP$ contains no subposet isomorphic to $(R,\leq_R)$.
Therefore, by Theorem \ref{thm:First-Fit-bounded}, there is an absolute constant $c$ such that $\Gamma(\posP)\leq c$.
\end{proof}
 
\begin{lem}
\label{lem:auxiliary_game_exponential_lower_bound}
Let\/ $a\geq 2$ and\/ $w\geq 2$.
Let\/ $\posP$ be a poset that consists of\/ $w$ chains\/ $C,C_1,\ldots,C_{w-1}$ such that for\/ $1\leq i\leq w-1$ and for every interval\/ $I$ of\/ $C$ of size\/ $n_i$, the chain\/ $C_i$ contains\/ $(a+1)k$ points\/ $x$ with\/ $I_\posP(x)\cap C=I$.
If\/ $k<(1+1/\lfloor\frac{a}{2}\rfloor)^{w-1}$ and\/ $|C|\gg n_1\gg\cdots\gg n_{w-1}\gg k$, then\/ $\vala(a,1,\posP)>k$.
Consequently, $\vala(a,1,w)\geq(1+1/\lfloor\frac{a}{2}\rfloor)^{w-1}$.
\end{lem}

\begin{proof}[Proof idea]
Bob plays like in the proof of Lemma \ref{lem:1_1_coloring_lower_bound} in $w-1$ phases, in the $i$th phase adding points of $C_i$ to $A$ until at least the $1/(\lfloor\frac{a}{2}\rfloor+1)$ fraction of the colors remaining after the previous phase have been used on $A\cap C_i$.
It follows that all colors are used on $A$ after all phases.
\end{proof}

\begin{proof}[Proof of Theorem \ref{thm:exponential_lower_bound}]
If $w\geq 2$ and $\frac{a}{b}\geq 2$, then $\valg(a,b,w)\geq\vala(\lfloor\frac{a}{b}\rfloor,1,w)\geq(1+1/\lfloor\frac{a}{2b}\rfloor)^{w-1}$, by Lemma \ref{lem:auxiliary_coloring_game}, Lemma \ref{lem:auxiliary_game_exponential_lower_bound}, and the fact that $\lfloor\frac{1}{2}\lfloor\frac{a}{b}\rfloor\rfloor=\lfloor\frac{a}{2b}\rfloor$.
\end{proof}

\section{Upper bound for the \texorpdfstring{$(2,1)$}{(2,1)}-coloring game}
\label{sec:upper_bound}

Let $w\geq 1$, $(C,\leq)$ be a chain, and $\scrI$ be a family of intervals of $(C,\leq)$ with no sequence of $w$ intervals $I_1\Subset\cdots\Subset I_w$, where $I\Subset J$ denotes that $I\subset J$ and $I$ does not contain the minimum and maximum points of $J$.
The \emph{$w$-game} on $(C,\leq)$ with interval set $\scrI$ and color set $\Gamma$ is played by two players, Presenter and Painter, who can take two kinds of actions.
One is that Presenter presents an interval $I\in\scrI$ and specifies a color $f(I)\in\Gamma$ called the \emph{forbidden color} for $I$.
The other is that Presenter or Painter assigns a color from $\Gamma\cup\{\delete\}$ to a point of $C$, where $\delete$ is a special color that cannot be forbidden for any presented interval.
Doing so, the players must obey the rule that the color of every colored point $x$ is different from $f(I)$ for any presented interval $I\ni x$.
Every round of the $w$-game proceeds according to one of the following scenarios:
\begin{enumerate}
\item\label{k-game:color} Presenter assigns an available color from $\Gamma\cup\{\delete\}$ to an uncolored point of $C$.
\item\label{k-game:interval} Presenter presents an interval $I\in\scrI$ and specifies its forbidden color $f(I)$.
Painter replies by assigning available colors from $\Gamma$ to $0$, $1$ or $2$ uncolored points of~$I$.
\item\label{k-game:point} Presenter picks an uncolored point of $C$ and asks Painter to color it.
Painter has to reply with an available color from $\Gamma$.
If no color is available, then Presenter wins the $w$-game.
\end{enumerate}
Painter wins the $w$-game when all points of $C$ have been colored.

\begin{lem}
Painter has a winning strategy in every\/ $w$-game played with\/ $2^{w-1}$ colors.
\end{lem}

\begin{proof}[Proof sketch]
The proof goes by induction on $w$.
The case $w=1$ is trivial, so let $w\geq 2$.
Painter plays the $w$-game with color set $\{1,\ldots,2^{w-2}\}\times\{0,1\}$ simulating the $(w-1)$-game with color set $\{1,\ldots,2^{w-2}\}$ on some intervals that are not $\Subset$-maximal and keeping the following invariant:
\begin{enumerates}
\item\label{invariant} For every $i\in\{1,\ldots,2^{w-2}\}$, if an interval $I$ (presented or not) has no points of colors $(i,0)$, $(i,1)$ and an interval $J$ has been presented with $f(J)=(i,j)$, $j\in\{0,1\}$, then either $I\cap J$\pagebreak[4] has no uncolored points or $I\Subset J$.
In the latter case, if $I$ has been presented with $f(I)=\linebreak[4](i,1-j)$, then $I$ has been passed to the simulated $(w-1)$-game with forbidden color $i$.
\end{enumerates}
\emph{Case \ref{k-game:color}}:
Presenter assigns color $(i,j)$ or $\delete$ to a point $x$.
Painter simulates Presenter's move in the $(w-1)$-game assigning color $i$ or $\delete$, respectively, to $x$.\\
\emph{Case \ref{k-game:interval}}:
Presenter presents an interval $I$ and specifies $f(I)=(i,j)$.
If no interval $J\Supset I$ has been presented before with $f(J)=(i,1-j)$, then Painter can (and does) assign color $(i,1-j)$ to at most two points in $I$ so that \ref{invariant} is preserved; Painter also simulates Presenter's move in the $(w-1)$-game assigning color $\delete$ to these points.
Now, suppose an interval $J\Supset I$ has been presented before with $f(J)=(i,1-j)$.
Painter simulates Presenter's move in the $(w-1)$-game presenting $I$ with forbidden color $i$ and observes Painter's strategy in the $(w-1)$-game coloring $0$, $1$ or $2$ points of $I$.
For each such point $x$ assigned color $i'$ in the $(w-1)$-game, Painter colors $x$ in the $w$-game using an available color $(i',j')$, which exists by \ref{invariant}.\\
\emph{Case \ref{k-game:point}}:
Presenter asks Painter to color a point $x$.
Painter simulates this move of Presenter in the $(w-1)$-game and observes Painter's strategy in the $(w-1)$-game assigning a color $i$ to $x$.
Painter responds in the $w$-game with an available color $(i,j)$, which exists by \ref{invariant}.
\end{proof}

\begin{proof}[Proof sketch of Theorem \ref{thm:upper_bound}]
Let $(P,\leq)$ be a poset of width $w$, $C_1,\ldots,C_w$ be a partition of $P$ into $w$ chains, and $\Gamma_1,\ldots,\Gamma_w$ be pairwise disjoint sets each of size $2^{w-1}$.
Let $\scrI_i=\{I_i(x)\colon x\in P\setminus C_i\}$, where $I_i(x)$ denotes the set of points in $C_i$ incomparable to $x$.
Alice plays on $P$ with color set $\Gamma_1\cup\cdots\cup\Gamma_w$ simulating the $w$-game played on $C_i$ with interval set $\scrI_i$ and color set $\Gamma_i$ for $1\leq i\leq w$.
If Bob assigns a color $\gamma\in\Gamma_i$ to a point $x\in C_i$, then Alice simulates Presenter's move \ref{k-game:color} on $C_i$ assigning color $\gamma$ to $x$.
If Bob assigns a color $\gamma\in\Gamma_i$ to a point $x\in C_j$, where $i\neq j$, then Alice simulates Presenter's move \ref{k-game:color} on $C_j$ assigning color $\delete$ to $x$ and Presenter's move \ref{k-game:interval} on $C_i$ presenting the interval $I_i(x)$ and specifying $\gamma$ as its forbidden color, and she replies according to Painter's response on $C_i$ assigning colors from $\Gamma_i$ to $0$, $1$ or $2$ points of $C_i$.
Painter's ability to reply in scenarios \ref{k-game:point} ensures that Alice can make `idle' moves (so as to color exactly $2$ points in each round) and Bob never achieves a winning position.
\end{proof}

\bibliographystyle{abbrv}
\bibliography{partitioning_game_bib}

\begin{thebibliography}{10}

\bibitem{Bod91}
H.~L. Bodlaender.
\newblock On the complexity of some coloring games.
\newblock {\em Int. J. Found. Comput. Sci.}, 2(2):133--147, 1991.

\bibitem{BKM13}
B.~Bosek, T.~Krawczyk, and G.~Matecki.
\newblock First-fit coloring of incomparability graphs.
\newblock {\em SIAM J. Discrete Math.}, 27(1):126--140, 2013.

\bibitem{DZ99}
T.~Dinski and X.~Zhu.
\newblock A bound for the game chromatic number of graphs.
\newblock {\em Discrete Math.}, 196(1--3):109--115, 1999.

\bibitem{FKKT93}
U.~Faigle, W.~Kern, H.~A. Kierstead, and W.~T. Trotter.
\newblock On the game chromatic number of some classes of graphs.
\newblock {\em Ars Combin.}, 35(17):143--150, 1993.

\bibitem{Gar81}
M.~Gardner.
\newblock Mathematical games.
\newblock {\em Sci. Amer.}, 244(4):18--26, 1981.

\bibitem{GZ99}
D.~J. Guan and X.~Zhu.
\newblock Game chromatic number of outerplanar graphs.
\newblock {\em J. Graph Theory}, 30(1):67--70, 1999.

\bibitem{HZ13}
F.~Havet and X.~Zhu.
\newblock The game {G}rundy number of graphs.
\newblock {\em J. Combin. Optim.}, 25(4):752--765, 2013.

\bibitem{Kie81}
H.~A. Kierstead.
\newblock An effective version of {D}ilworth's theorem.
\newblock {\em Trans. Amer. Math. Soc.}, 268(1):63--77, 1981.

\bibitem{Kie00}
H.~A. Kierstead.
\newblock A simple competitive graph coloring algorithm.
\newblock {\em J. Combin. Theory Ser. B}, 78(1):57--68, 2000.

\bibitem{Kie05}
H.~A. Kierstead.
\newblock Asymmetric graph coloring games.
\newblock {\em J. Graph Theory}, 48(3):169--185, 2005.

\bibitem{KT94}
H.~A. Kierstead and W.~T. Trotter.
\newblock Planar graph coloring with an uncooperative partner.
\newblock {\em J. Graph Theory}, 18(6):569--584, 1994.

\bibitem{KT96}
H.~A. Kierstead and Z.~Tuza.
\newblock Game coloring numbers and treewidth.
\newblock manuscript.

\bibitem{KY05}
H.~A. Kierstead and D.~Yang.
\newblock Very asymmetric marking games.
\newblock {\em Order}, 22(2):93--107, 2005.

\bibitem{WZ08}
J.~Wu and X.~Zhu.
\newblock Lower bounds for the game colouring number of partial {$k$}-trees and
  planar graphs.
\newblock {\em Discrete Math.}, 308(12):2637--2642, 2008.

\bibitem{YZ08}
D.~Yang and X.~Zhu.
\newblock Activation strategy for asymmetric marking games.
\newblock {\em European J. Combin.}, 29(5):1123--1132, 2008.

\bibitem{Zhu99}
X.~Zhu.
\newblock The game coloring number of planar graphs.
\newblock {\em J. Combin. Theory Ser. B}, 75(2):245--258, 1999.

\bibitem{Zhu00}
X.~Zhu.
\newblock The game coloring number of pseudo partial {$k$}-trees.
\newblock {\em Discrete Math.}, 215(1--3):245--262, 2000.

\bibitem{Zhu08}
X.~Zhu.
\newblock Refined activation strategy for the marking game.
\newblock {\em J. Combin. Theory Ser. B}, 98(1):1--18, 2008.

\end{thebibliography}

\end{document}